\documentclass{article}
\RequirePackage{fix-cm}
\usepackage{tikz}
\usepackage{tikz-qtree}
\usetikzlibrary{trees}

\usepackage{setspace}

\usepackage[misc,geometry]{ifsym} 

\usepackage{color}
\usepackage{srcltx}
\usepackage{fancyhdr}
\usepackage{latexsym}
\usepackage{amsfonts}
\usepackage{amscd}
\usepackage{epsfig}
\usepackage[cp850]{inputenc}
\usepackage[english]{babel}
\usepackage{ifthen}
\usepackage{float}
\usepackage{amsmath}
\usepackage{amssymb}
\usepackage{mathrsfs}
\usepackage[mathscr]{eucal}
\usepackage{dsfont}
\usepackage{amstext}
\usepackage{syntonly}

\usepackage{tikz} 
\usepackage{color}
\usepackage{graphicx}

\usepackage{enumerate}
\usepackage{amssymb}
\usepackage[english]{babel}
\usepackage{amsthm}
\usepackage{amstext}
\usepackage{graphicx}
 \usepackage{mathptmx}      
\usepackage{latexsym}
  \usepackage{enumerate}

\newtheorem{theorem}{Theorem}[section] 

\newtheorem{proposition}[theorem]{Proposition}
\newtheorem{definition}[theorem]{Definition}   
\newtheorem{remark}[theorem]{Remark}
\newtheorem{example}[theorem]{Example}

\usepackage{tikz}
\usetikzlibrary{graphs,graphs.standard,quotes}
\usepackage{color}

\usepackage{enumerate}
\usepackage{amssymb}
\usepackage[english]{babel}
\usepackage{amstext}
\usepackage{qtree}
\usepackage{graphicx}
 \usepackage{mathptmx}      
\usepackage{latexsym}
\newcommand{\footremember}[2]{%
    \footnote{#2}
    \newcounter{#1}
    \setcounter{#1}{\value{footnote}}%
}

\date{}
\author{%
Juan M. Benito-Ostolaza\footremember{alley}{Institute for Advanced research in Business and Economics (INARBE) \& Dpto. Econom\'{i}a, Universidad P\'{u}blica de Navarra (UPNA). Campus Arrosad\'{\i}a, 31006.  Iru\~na-Pamplona, Navarre, Spain.} 
\and
  Mar\'ia Jes\'us Campi\'on \footremember{alleydd}{Institute for Advanced research in Business and Economics (INARBE)\&  Dpto. Estad\'istica, Inform\'atica y Matem\'aticas,  Universidad P\'{u}blica de Navarra (UPNA). Campus Arrosad\'{\i}a, 31006.  Iru\~na-Pamplona, Navarre, Spain.}
  \and 
   Asier Estevan\footremember{trailer}{Institute for Advanced Materials and Mathematics (INAMAT)\&  Dpto. Estad\'istica, Inform\'atica y Matem\'aticas,  Universidad P\'{u}blica de Navarra (UPNA). Campus Arrosad\'{\i}a, 31006.  Iru\~na-Pamplona, Navarre, Spain.    }
  }
  
\title{A mathematical approach to law and deal modelling: legislation and agreements} 

\begin{document}
\maketitle
\date{\today}
\begin{abstract}

This paper presents a new and original first approach to agreement situations as well as to regulations constructions. This is made by giving a mathematical formalization to the set of all possible deals or regulations, such that then, the proximity between them is defined by means of a premetric. Thanks to this mathematical structure that tries to capture the problematic of agreements and regulation modifications, now some questions related to game theory or law are reduced to mathematical optimization problems.  

\end{abstract}



\section{Motivation. The need of going further when making agreements} \label{s1}

The objective of legal systems around the world is to regulate 
people's coexistence (\cite{Mattei}). This objective gives a central role to the definition and measurement of the legal framework in everyday interactions, as well as to the way in which these norms regulate the behavior of a population. 
In this sense, over the last decades, a good deal of theoretical and empirical work has been devoted to the task of explaining the compositition of agreements and consensus in a lot of related areas, such as climate change (\cite{Ore, Walter}), goverments coalitions (\cite{Laver98, Warwick}), cooperation and coordination on social dilemmas (\cite{VanL2013}) among others.
Likewise, the evolution of large-scale cooperation in humans has been the focal point of several disciplines in the last few decades (\cite{Axelrod84, Bowless2011,  Boyd85}). It is argued that genetic relatedness can explain the emergence of cooperation (\cite{Hamilton64}). However, humans cooperate on vast scales, more substantially than their genetic relatedness allows. Some explanations for this phenomenon include punishment, direct and indirect reciprocity, group competition etc. (see, for instance, \cite{Henrich2001, Young2015, Yu2015}).
Furthermore, the socio-economic spectrum in which we live is replete with binding agreements and/or treaties between two or more countries, companies, political parties, etc.

In the international environment, in recent years, there has been a proliferation of free trade agreements, such as the European Union, the Trans-Pacific Partnership Treaty or the Transatlantic Trade and Investment Partnership, and even environmental treaties and protocols such as the Kyoto Protocol.  But, what makes some countries sign these treaties and others not?, and what makes countries comply with the agreements or not? Undoubtedly the regulation distance within the agreed agreements and what each participant is willing to accept or renounce makes the difference between signing and not signing. Most of the proposals to solve these differences are based on fragmenting the regulations imposed in the agreements to satisfy the greatest number of participants. \cite{Biermann09} note that this type of fragmentation of the legal framework promotes the multiplication of international commitments. But, the difference in the starting point of the participants in the treaties may or may not make the agreement a success. 

This situation of non-compliance does not only occur in international environments. Within a country's borders, undesirable situations can arise with respect to compliance with national regulations, even in public health situations. One of the great battles of modern jurists and politicians is how we should implement the enforcement of laws. Given an ultimate goal, should we implement all the laws or rules to achieve it simultaneously, or does a gradual sequencing in the introduction of the different rules cause the number of non-compliances to decrease? Take for example the recent COVID19 global pandemic case. All countries, even regions, have adopted different measures in both form and timing. All measures have been designed to contain the spread of the virus and/or to control space in hospitals to care for the sick. The way these measures have been implemented has varied widely, and it has been found that even in this pandemic situation, the measures adopted have been breached by many residents. Perhaps a gradual implementation of the different measures could have resulted in less non-compliance, thus achieving greater effectiveness of the measures adopted.

In this work, we show that phasing in the rules to be complied with can achieve the objective at a lower cost than if they were all implemented from the beginning. The higher or lower cost depends on the \emph{proximity} between rules, which is defined in the paper (mainly through the probability and `severity' of breaking a rule), and which also allows us to evaluate what kind of agreements or treaties can be signed by the participants, taking into account this proximity between the different rules or laws.

 Consequently, social scientists, practitioners and legal researchers have been concerned with developing the legal framework of people. Although much work has been done to develop the regulatory character of the legal system, we see that little has been done in terms of developing methods to measure the status of people's legal system. Therefore, we see a need to measure and quantify this legal environment and the distance between existing laws and agreements.
 The aim of this paper is, using a new information theory approach to compare and evaluate different legal measures, to derive some focal measures of the effect of the legal framework on individuals and populations.
 In addition, in this paper, we contribute to the literature by analyzing how the proximity between diferent positions and norms in agreements affects the evolution of consensus or adoption of the agreement.

The goal of the present work is to introduce a tool in order to measure the proximity or `distance' between games, regulations or deals. One of the main aims of this study  is to know how to change the rules of a game in order to achieve a more efficient equilibrium. There may be different manners or paths in order to change the law from an initial regulation to a final one.
Here we construct a tool that indicates which is the path when changing the law, that is, a tool that shows us how we have to make changes in the law in order to arrive to a better regulation or deal, but always minimizing the \emph{social resistence}  to the changes.
The second goal of the paper is to achieve the \emph{most preferred} deal in no-agreement situations. Again, for that purpose a  distance function between  deals is required.

The introduced technique allows us to achieve the closest agreement-point in case of blocking situation.

As far as we know, there is no previous knowledge on the subject, so this is the first mathematical approach to the problem posed.
\bigskip

The rest of the paper is organized as follows. Section \ref{preliminaries} introduces the mathematical tools and notation needed. Section  \ref{Sdistances}  gives a mathematical structure to the problematic of the agreements and regulations, so that then, Section~\ref{sapply} presents some applications through the introduced new approach.
A final section of conclusions ends the paper.

\section{Preliminaries: Mathematical tools}\label{preliminaries}

In the present section we include some mathematical concepts that shall be needed throughout the paper. These are mainly related to order structures and metric spaces.

 From now on $X$ will stand for a nonempty set.
 
  \begin{definition} \rm A binary relation $\mathcal{R}$ on $X$ is a subset of the cartesian product $X \times X$. Given two elements $x,y \in X$, we will use the standard notation $x \mathcal{R} y$ to express that the pair $(x,y)$ belongs to $\mathcal {R}$.

  Associated to a binary relation $\mathcal{R}$ on a set $X$, we consider its \emph{negation} \rm (respectively, its \emph{transpose}\rm) as the binary relation $\mathcal{R}^c$ (respectively, $\mathcal{R}^t$) on $X$ given by $(x,y) \in \mathcal{R}^c \iff (x,y) \notin \mathcal{R}$ for every $ x,y \in X$ (respectively, given by $(x,y) \in \mathcal{R}^t \iff (y,x) \in \mathcal{R}, \ $ for every $ x,y \in X)$. We also define the \emph{adjoint} \rm $\mathcal{R}^a$ of the given relation $\mathcal{R}$, as $\mathcal{R}^a = (\mathcal{R}^t)^c$. 

  A binary relation $\mathcal{R}$ defined on a set $X$ is said to be: 
  \begin{itemize}
  \item[(i)] \emph{reflexive} \rm if $x \mathcal{R} x$ holds for every $x \in X$,
  \item[(ii)] \emph{irreflexive} \rm if $\neg(x \mathcal{R} x)$ holds for every $x \in X$, 
  \item[(iii)] \emph{symmetric} \rm if $\mathcal{R}$ and $\mathcal{R}^t$ coincide,
  \item[(iv)] \emph{antisymmetric} \rm if $\mathcal{R} \cap \mathcal{R}^t \subseteq \Delta = \{ (x,x): x \in X \}$,
  \item[(v)] \emph{asymmetric} \rm if $\mathcal{R} \cap \mathcal{R}^t = \varnothing$,
  \item[(vi)] \emph{total} \rm if $\mathcal{R} \cup \mathcal{R}^t = X \times X$,
 \item[(vii)] \emph{transitive} \rm if $x \mathcal{R} y \wedge y \mathcal{R} z \Rightarrow x \mathcal{R} z $ for every $x,y,z \in X$.
 \end{itemize}
 \end{definition} 

 In the particular case of a nonempty set where some kind of \emph{ordering}   has been defined, the standard notation is different. We include it here for sake of completeness, and we will use it throughout the present manuscript.

   \begin{definition} \rm 
   A \emph{preorder} $\precsim$ on $X$ is a binary relation  which is reflexive and transitive. An antisymmetric preorder is said to be an \emph{order}. A \emph{total preorder} \rm $\precsim$ on a set $X$ is a preorder such that if $x,y \in X$ then $(x \precsim y) \vee (y \precsim x)$ holds. In the case of preorders, it is well known that the corresponding indifference is transitive, i.e., it is an equivalence relation. Given a preorder $\precsim$ on $X$, we shall denote by  $x\bowtie y$ when $x$ and $y$ are not comparable, i.e., when $\neg(x\precsim y)$ as  well as $\neg(y\precsim x)$. 
 \end{definition}

 \begin{definition}\rm
  Let $\prec$ denote an
asymmetric binary relation on $(X, \tau)$.   Given $a \in X$ the sets
$L_{\prec}(a) = \lbrace t \in X \ : \ t
\prec a \rbrace $ and $
U_{\prec}(a) = \lbrace t \in X \ :
 \ a \prec t \rbrace $ are called, respectively, the \emph{strict lower and upper contours}
 \rm of $a$ relative to $\prec$. We say that $\prec$ is  \emph{$\tau$-continuous} \rm (or just \emph{continuous})
if for each $a \in X$ the  sets $L_{\prec}(a)$ and $U_{\prec}(a)$
are $\tau$-open.

We will denote the \emph{order topology} generated by $\prec$ as $\tau_{\prec}$, and it is defined by means of the  subbasis provided by the lower and upper contour sets.

Let $\precsim$ denote a
reflexive binary relation on $(X, \tau)$.   Given $a \in X$ the sets
$L_{\precsim}(a) = \lbrace t \in X \ : \ t
\precsim a \rbrace $ and $
U_{\precsim}(a) = \lbrace t \in X \ :
 \ a \precsim t \rbrace $ are called, respectively, the \emph{weak lower and upper contours}
 \rm of $a$ relative to $\precsim$. We say that $\precsim$ is  \emph{ $\tau$-lower semicontinuous} \rm (\emph{$\tau$-upper semicontinuous})
if for each $a \in X$ the  sets $L_{\precsim}(a)$ (resp. $U_{\precsim}(a)$)
are $\tau$-closed.
\end{definition}

The concept of $\tau$-(semi)continuity is a key notion when dealing with continuous representations of, not only in the well-known case of total preorders (see \cite{BRME, de, Debr}), but also preorders, thus, it is a very common and usually necessary concept in the field of order structures. 



Now, we include some definitions related to metrics and their generalizations.


\begin{definition}\rm
A \emph{metric} on a set $X$ is a function $d\colon X\times X\to [0,+\infty)$ (also called \emph{distance})
satisfiying the following three properties:

$%
\begin{array}
[c]{l}%
\mathrm{(M1)}$\textrm{\ }$d(x,x)=0 \, \text{ for all }\, x\in X. \\

\mathrm{(M2)}$\textrm{\ }$d(x,y)=0 \Rightarrow x=y.\\
\mathrm{(M3)}$\textrm{\ }$d(x,y)=d(y,x).\\
\mathrm{(M4)}\text{\textrm{\ }}d(x,z)\leq d(x,y)+d(y,z).
\end{array}
\medskip$
\end{definition}

There are several generalizations of metrics (see \cite{top1, bana, top2}) relaxing the aforementioned axioms. 
In our context we shall use the following notions.

\begin{definition}\rm
 A nonnegative real-valued function $d$ on
$X\times X$ is a \textit{quasi-semi distance} if    for all $x,y,z\in X:\medskip$

$%
\begin{array}
[c]{l}%
\mathrm{(M1)}d(x,x)=0.\\
\mathrm{(M4)}\text{\textrm{\ }}d(x,z)\leq d(x,y)+d(y,z).
\end{array}
\medskip$

A \textit{premetric}  on a set $X$ is a nonnegative real-valued function $d$ on
$X\times X$ such that :\medskip

$
\begin{array}
[c]{l}%
\mathrm{(M1)}$\textrm{\ }$d(x,x)=0 \,\text{ for all } x\in X.\\
\end{array}
\medskip$

Given a function $ d: X\times X \rightarrow \mathbb{R}^{+}$
we will define $d^{-1}(x, y) = d(y, x)$ and
$d^s(x, y) = max\{d(x, y), d(y, x)\}$. Notice that, if $d$ is a quasi-semi distance, $d^{-1}$ is a quasi-semi distance and $d^s$ is a distance.
\end{definition}

The following proposition is well-known  and it warrants that any premetric defines a topology on the set. \cite{top1,bana,top2}

\begin{proposition}\label{Popenballs}
Let $d$ be a premetric on $X$. Then, $d$ induces a topology on $X$ by means of the following open balls, for any $r>0$ and any $x\in X$:
$$ B_r(x)=\{y\in X \colon d(x,y)<r\}. $$

\end{proposition}

Finally, we include one statistical distance which  quantifies the proximity between two  probability distributions. There are several distances defined between probability distributions, in the present paper we shall work with a $f$-divergence function  known as \emph{Kullback Leibler divergence}. \cite{aman, norma, jon}



\begin{definition}\label{Dkl}\rm
Let $p$  and $q$ be two probability functions on the  same event space $\mathcal{F}$. The \emph{Kullback Leibler divergence} (also called \emph{relative entropy}) from $q$ to $p$ (or the \emph{relative entropy} from $q$ to $p$)
is defined by
$$ D(p||q)= \displaystyle \sum\limits_{x\in \mathcal{F}} p(x)\cdot \log (\frac{p(x)}{q(x)})=E_{p}\log(\frac{p(X)}{q(X)}).$$

Relative entropy is defined whenever condition 
 $$q(x)=0 \implies p( x ) = 0, \quad x\in X\qquad \hfill (1)$$
 is satisfied.

 The \emph{entropy} of $p$ is defined by
 $$H(p)=-\displaystyle \sum\limits_{x\in \mathcal{F}} p(x)\cdot \log(p(x)).$$
\end{definition}

\section{Comparing   regulations and societies}\label{Sdistances}
Through the present paper we introduce a tool that allows to quantify negotiations. It gives us a manner to objectify agreements from  a theoretical point of view, that could help for impartiality, in particular when there is an `oracle' or `institution' that want to achieve an agreement between the agents. This new technique allows us  to find the \emph{most preferred} deal in disagreement situations.
Furthermore, 
if we want to incorporate a regulation by adding rules at different times, this method gives us the best manner or order in wich the rules must be added so that they clash as little as possible with the custom of society (as it is shown in Example~\ref{Ecom}).  Thus, we may construct the best (or `less aggressive') path from a deal to a new one. This may be interesting too in order to keep the consensus.


For these purposes, we need to formalize some ideas by attributing them a mathematical definition. As far as we know, there is no previous knowledge on the subject, so this is the first mathematical approach to the problem posed.
\medskip

First, let $\Omega$ be a set of rules $\{R_i\}_{i\in I}$  called \emph{Law}, that is used  when defining a \emph{regulation} or \emph{legal framework} $G$, which is in fact a subset of $\Omega$.  Second, let  $\{A_n\}_{n=1}^N$  be a finite family of agents or players, these players will be the signers of the possible deal or the society that lives under the regulation $G$. The set of rules $G$, that is, the regulation,  will be the sample space $\Omega_G$ of a probability space,  where the events space is the $\sigma$-algebra $\mathcal{F}=\mathcal{P}(G)$, which is the set of subsets of rules of the regulation that could be broken.

  


If we change a rule, two  interesting   questions arise:

\begin{enumerate}
\item[$(Q1)$] What is the probability that the rule will be broken by an agent $A$? Or, what is the probability that the rule will be broken by $n$ players?
\item[$(Q2)$]  Which is the minimum punishment that must be imposed in order to avoid breaking the rule?   Or, in other words,  how much would have to be compensated in order to guaranty that the rule will be cumplied?
\end{enumerate}

 The answer  to question 2  would give us the \emph{individual cost or severity} of enforcing the rule, and with the answer  to question 1 we will be able to measure the  proximity between the regulation of the deal or game with the habits of the current society of players.

We assume that we have this information at our disposal. From an applied point of view, it may be estimated, or it may be deduced from historical data. In any case, from a theoretical point of view, we will need this information (or the corresponding estimation, or frequencies or approximation).

Furthemore, the answer to both questions may be related. We may think that if a rule is broken usually by plenty of players, then the individual cost or severity must be small (for example, if we are thinking on forbidden smoke somewhere), whereas if a rule is rarely broken (a murder, major tax fraud) then the individual cost is bigger. Hence, this individual cost or severity may be defined by means of the probability of being broken.


Now, we formalize the main concepts. 


\begin{definition}\rm
Let $\Omega$ be  a set of rules $\{R_i\}_{i\in I}$  called \emph{Law} and $A=\{A_n\}_{n=1}^N$ a finite set of players. A subset  $G=\{R_k\}_{k\in K\subseteq I}\subseteq \Omega$ is said to be a \emph{regulation} or \emph{legal framework} under the Law. 
For a given regulation $G$, $G$ itself is  a sample space $\Omega_G$, where the corresponding \emph{events space} is defined by the set of subsets $\mathcal{F}_G=\mathcal{P}(G)$. An \emph{event} $x\in \mathcal{F}_G $ is a set of rules that may be broken (each of one). Thus, 
for each $n\in \{1,2,..., N\}$, let
$P_n \colon \mathcal{F}_G\to [0,1]$ be the probability of player $A_n$ breaking a subset of rules $x\in \mathcal{F}_G $ (i.e., breaking each of the rules involved in $x$). 
%
%
%
 By means of these probabilities, we define the mean function $P\colon \mathcal{F}_G\to [0,1]$ of  \emph{expected ratio of  offenders} by
$P(x)=\frac{1}{N}\sum\limits_{n=1}^N  P_n(x)$. Thus, the expected number of offenders for a subset of rules in a regulation (that is, for an event) $x\in \mathcal{F}_G$ is $N\cdot P(x)$. Let $g\colon  \mathcal{F}_G\to \mathbb{R}$ the \emph{(individual) punishment} associated to  an event $x\in \mathcal{F}_G$ 
(i.e., associated to breaking each of the rules involved in $x$).

A \emph{legal game} (or \emph{l-game}, for short) $\Psi_G$ 
  is now a   probabilistic space $(\Omega_G,\mathcal{F}_G, P)$ associated to a legal framework
  $G$ (defined under  the Law $\Omega$) 
and a society or family of agents $A$. The idiosyncrasies of this society $A$ are reflected in the probability functions $P$. Thus, the same regulation applied on another society $B$ implies a different probability function $P'$. If it is also endowed with a punishment function $g$, then we say that it is a \emph{punished l-game}.

From now, $\Gamma$  denotes the set of all possible l-games that may be defined with a given  Law $\Omega$ on a society $A$.
Thus, $\Gamma=\{(G, \mathcal{F}_G, P_G)\, \colon \, G\subseteq \Omega \}$, and we shall call it by \emph{Game of games}. Therefore, when referring to $\Gamma$,  the Law and the society are  assumed to be  fixed.
We shall say that a rule  $R_i\in \Omega$   is \emph{titere} if $P(x)=0$, for any $x\in \mathcal{F}_G$ such that $R_i\in x$.

Then, the \emph{expected severity of a punished l-game} (per capita) is defined by 
$$E g(\Psi_G)=\displaystyle \sum\limits_{x\in \mathcal{F}_G}  P(x)\cdot g(x).$$
 Thus, the expected severity of the l-game is the sum of the products of the expected  ratio of offenders  and the individual cost or severety.

 The \emph{entropy} \cite{aman, norma, jon} of a (not necessarily punished) l-game is
 $$H(\Psi_G)=-\displaystyle \sum\limits_{x\in \mathcal{F}_G} P(x)\cdot \log(P(x)).$$ 
 Thus, the entropy may be interpreted too as the expected severity per capita when $g(x)=\log(\frac{1}{P(x)})$, for any $x\in \mathcal{F}_G$.
 
 By continuity, it is assumed that $0\cdot \log(0)=0$ and $0\cdot \log(\frac{1}{0})=0$.

\end{definition}
\begin{remark}\rm

\noindent$(1)$ Notice that for a given regulation $G\subseteq \Omega$, $\sum\limits_{x\in \mathcal{F}_G}  P_n(x)=1$ for each $n\in \{1,2,..., N\}$ and, hence, $\sum\limits_{x\in \mathcal{F}_G}  P(x)=1$.
Therefore, function $P$ is, in fact, a probability function. In the case where the signers are countries or a few players, $P$ may be defined by means of the individual's probability functions $P_n$ ($n=1,...,N$). On the other hand, for the case of regulations that are applied on societies (i.e. with thousands of persons  involved), funtion $P$ maybe directly estimated or defined by means of frequency data, without requiring the definitions of the individual's probability functions $P_n$ ($n=1,...,N$).
It is interesting that, for the first case aforementioned, this function $P$ may be defined with weights, such that the weigth of some players may be bigger than others regarding to the their importance in the game (for example, if we want to weigth a family of countries in a deal by means of their GDP, population, contamination, ...)

\medskip

\noindent$(2)$ The (individual) punishment $g(x)$ (for any $x\in\mathcal{F}_G$) may be assumed to be in $[0,1]$, although it can be defined on the   set of positive real numbers. It is assumed that $g(\emptyset)=0$, that is, there is no punishment when there are no offenders.

Notice that, if the individual punishment associated to breaking a family of rules is defined by means of the probability function $P$ as $ \log(\frac{1}{P(x)})$ (such that, the punishment also changes depending on the number of offenders), then the expected severity per capita is defined by the   entropy:



$$E g(\Psi_G)=\displaystyle \sum\limits_{x\in \mathcal{F}_G} P(x)\cdot \log(\frac{1}{P(x)}).$$

\end{remark}

\vspace{0.5cm}

From now on, we will forget the player's probability functions $P_n$ (for each $n=1,...,N$) and we will work just with the (weighted) mean $P$.

Notice that a legal game consists on a set of rules applied on a   society, the idiosyncrasy  of that society is reflected in the probability function $P$. 
Thus, the l-game arised from a regulation $G$ in a society $A$ may change just by changing the society, i.e., the consecuences of a regulation may differ from one society to another.
In this line,  we introduce the idea of \emph{distance} between societies  in the spirit of Kullback Leibler divergence. \cite{aman, norma, jon} 
It will give us the behavior difference between two societies which live under the same regulation. What it works in a society, may fail in another.

\begin{definition}\label{Dsocial}\rm
Let $\Psi_G$ and $\Phi_G$ be two legal games with the same regulation but with different societies $A$ and $B$, which are characterized by a different probability function, $p$ and $q$ respectively. The \emph{KL-social divergence} from $A$ to $B$ under $G$ (denoted by $D_G(A||B)$) is the Kullback Leibler divergence from $q$ to $p$, that is, 
$$ D_G(A||B)=D(p|| q) = \displaystyle \sum\limits_{x\in \mathcal{F}_G} p(x)\cdot \log (\frac{p(x)}{q(x)})=E_{p}\log(\frac{p(X)}{q(X)}).$$

\end{definition}
 
\begin{remark}
\noindent$(1)$
This mathemathical concept of \emph{social divergence} may be useful when trying to import/export a regulation from one country or region to another. A successful regulation may fail when it is applied in another society. Thus, it may be interesting to measure the proximity between those societies as regards to the regulation.  Hence, before trying to apply that new regulation, it may be important to minimize the distance between societies (as regards to the corresponding regulation) by means of an adequate punishment function, awareness campaigns, incentives, education, ...
 
\medskip

\noindent$(2)$
The idea of KL-social divergence has been defined by means of Kullback Leiber divergence, this explains the initials KL. However, as explained in Preliminaries, there are several divergences or distances that measure how different two probability distributions are from each other. Depending on that selection, we may define one social divergence or another. However, all these distance should satisfy some \emph{coherence} conditions in order to faithfully reproduce the problem in our mathematical model. 
\end{remark}

\bigskip

In the following lines we  present a distance function that allows us to measure the difference between l-games, in particular, between regulations that are applied on the same society. Before that, we collect a few conditions that should  be satisfied by such a distance. By means of these conditions, we want to traslate the main  properties of a negotiation or legislation to our mathematical model. 

\begin{definition}\rm\label{Dcoherence}

Let $\Omega$ be a Law and $A$ a society. Let $G$ be a regulation in $\Omega$ (thus, $G\subseteq \Omega$) and $R$ a rule in $\Omega\setminus G$. The l-games arised from $G$ and $G\cup \{R\}$ are denoted by $\Psi$ and $\Psi'$, respectively. Let $d$ be a distance function on the set of l-games $\Gamma$.  Then, the following \emph{coherence conditions} are defined:

\begin{enumerate}
\item[($C1$)] The bigger the severity of rule $R$, bigger the distance from $\Psi$ to $\Psi'$.

\item[($C2$)] The more likely it is that the rule $R$ will be breached, the greater the distance  from $\Psi$ to $\Psi'$.
\end{enumerate}
\end{definition}



\begin{definition}\rm\label{Ddis}

Let $\Psi_1$ and $\Psi_2$ be two punished l-games defined under a Law $\Omega$ associated to two regulations $G_1$ and $G_2$, respectively, with the corresponding probability functions $p_1$ and $p_2$ and the punishment functions $g_1$ and $g_2$, respectively.  
Then, the \emph{premetric from $\Psi_2$ to $\Psi_1$ } is defined by 
$$D(\Psi_1|| \Psi_2)=\displaystyle \sum\limits_{x\in \mathcal{F}(G_1)} p_1(x)\cdot |g_1(x)-g_2(x)|,$$
where $g_2(x)$ is 0 whenever $x\notin \mathcal{F}_{G_2}$. 
\end{definition}

\begin{remark}\rm

\noindent$(1)$
First, notice that the premetric function $d$ satisfies the coherence conditions introduced in Definition~\ref{Dcoherence}.
\medskip

\noindent$(2)$
With definition above, notice that $D(\Psi_1|| \{\emptyset\})=E_p g(\Psi_1)$, whereas $D( \{\emptyset\} || \Psi_1)=0$, for any l-game $\Psi_1$.
$D(\Psi_1||\Psi_2)$ may be read as the `social resistance' to pass from $\Psi_2$ to $\Psi_1$. 
Thus, the value $D(\Psi_1|| \{\emptyset\})=E_p g(\Psi_1)$ may be interpreted as the `severity', `cost' or `social resistence' of the regulation   $G_1$ when it is   implemented from scratch.  On the other hand, $D( \{\emptyset\} || \Psi_1)=0$ means that removing rules does not imply a `severity', `cost' or `social resistence'.  In order to impose an strictly positive value to $D( \{\emptyset\} || \Psi_1)$, the distance function can be symmetrized as suggested in the next point.
 
\medskip

\noindent$(3)$
This  premetric  fails to be symmetric, however, it can be easily symmetrized as follows:
$$ D^+(\Psi_1,\Psi_2)=D(\Psi_1||\Psi_2)+D(\Psi_2||\Psi_1).$$

Or, by means of the maximum:
$$ D^s(\Psi_1,\Psi_2)=\max\{D(\Psi_1||\Psi_2),D(\Psi_2||\Psi_1)\}.$$

\medskip

\noindent$(4)$
However, $D$ as well as $D^+$ and $D^s$ may fail to satisfy the triangle inequality. 
\end{remark}

\begin{theorem}
The distance function $D$ is a  premetric on $\Gamma$ and, hence, it defines  a topology $\tau_D$ on the set.

The distance functions $D^+$ and $D^s$ are symmetric premetrics on $\Gamma$  and, hence, each of them  defines a topology ($\tau_{D^+}$ and $\tau_{D^s}$, respectively) on the set.
\end{theorem}
\begin{proof}
It is straighforward to check that $D$ is in fact a premetric and, therefore, $D^+$ and $D^s$ are symmetric premetrics on $\Gamma$. Hence, by Proposition~\ref{Popenballs}, they define a topology on $\Gamma$ by means of the corresponding open balls.
\end{proof}

Notice that, through the definitions included until this point, now our set of all possible regulations in a society under a Law (assuming that each regulation is endowed with the corresponding probability function related to the society) is a complete and weighted and directed graph of $2^{|\Omega|}=|\mathcal{P}(\Omega)|$ nodes, where each node is a l-game  and each edge $(\Psi_1,\Psi_2)$ is weighted by the value determined by the function $D(\Psi_2,\Psi_1)$. If $D$ is symmetric, then the graph can be considered undirected. Here, notice that not just any arbitrary combination of rules can make (legal) sense, thus, from an applied point of view, the number of possible regulations arised from the Law $\Omega$ may be smaller that $2^{|\Omega|}$. 

If necessary, we may use the \emph{path-distance} arised from the weighted graph. It is known that, if the graph is connected and undirected, then this path-distance is in fact a distance and, hence, it satisfies the triangle inequality. In the case of a directed and complete graph, then the path-distance fails to be symmetric and, hence, it is a quasi-metric. \cite{graph}

\begin{theorem}
Let $D$ be a quasi-premetric on $\Gamma$. Hence, $\Gamma$ 
is a complete and directional graph where the weigth of each edge $(v,u)$ corresponds to $D(u,v)$ (which may differ from $D(v,u)$ if $D$ fails to be symmetric). Then, \emph{path-distance} $D_p$ is a (quasi) metric  on $\Gamma$  where $D_p(u,v)$ is defined as  the shortest distance between them where distance is measured  along the edges of the graph. In particular, $D_p$ satisfies the triangle inequality.
\end{theorem}



 
\begin{remark}\rm
Again,  if the individual punishment associated to breaking a family of rules is defined by means of the probability (such that the punishment also changes depending on the number of offenders), then the \emph{distance from $\Psi_2$ to $\Psi_1$} maybe defined by means of the \emph{joint entropy}. \cite{aman, norma, jon}
That could be useful  too when comparing societies which are not under exactly the same regulation.
%
\end{remark}



\begin{definition}\label{Dsubgame}\rm

Let $\Omega$ be  a Law of a finite set of possible rules and $A=\{A_n\}_1^N$ a finite set of players. Let $G_1$ and $G_2$ be two regulations under the Law, and    $p_1$ and $p_2$ the corresponding probabilities (respectively) arised when they are applied in a society $A$. We denote by $\Psi_1$ and $\Psi_2$ the l-games $(G_1, \mathcal{F}_1, p_1)$ and $(G_2, \mathcal{F}_2, p_2)$, respectively.
 We shall say that $\Psi_2$ is a \emph{l-subgame} of $\Psi_1$ (and denote it by $\Psi_2 \sqsubseteq_{sub} \Psi_1$) if $G_1\subseteq G_2$ and any titere of $G_2$ is also a titere in $G_1$.  
 
 We denote by $sub(\Psi_1,\Psi_2) $ the set of l-subgames of $\Psi_1$ and $\Psi_2$ (it may be empty).
\end{definition} 

\begin{remark}
\noindent$(1)$
With definition above, notice that if we add a rule to a game, then we achieve a new subgame of the first one, i.e., when adding rules we define subgames.

\noindent$(2)$ When the Law $\Omega$ and the society $A$ are fixed, we extend the  notation $D(\Psi_1||\Psi_2)$ between games to the set of regulations $\mathcal{P}(\Omega)$ by $D(G_1||G_2)=D(\Psi_1||\Psi_2)$, where $\Psi_1$ and $\Psi_2$ are the corresponding l-games arised from $G_1$ and $G_2$ (respectively) when they are applied in the society $A$.
\end{remark}

\begin{proposition}
The relation $\sqsubseteq_{sub}$ of being a l-subgame is a preorder relation on $\Gamma$.
\end{proposition}
\begin{proof}

It is clear that $\sqsubseteq_{sub}$ is reflexive. Let's see that it is transitive. Suppose that $\Psi_3 \sqsubseteq_{sub} \Psi_2\sqsubseteq_{sub} \Psi_1$, then  $G_1\subseteq G_2\subseteq G_3$ and, in particular, $G_1\subseteq G_3$. On the other hand, any titere of $G_3$ is also a titere in $G_2$, as well as any  titere of $G_2$ is also a titere in $G_1$, thus, any titere of $G_3$ is also a titere in $G_1$.

Therefore, we conclude that  $\Psi_3\sqsubseteq_{sub} \Psi_1$, that is,  $\sqsubseteq_{sub} $ is transitive and, hence, it is a preorder.
\end{proof}











\begin{example}\label{Ecom}
Let $\{A_i\}_{i=1,2,3}$ be a set of three players, which are the herdsmans of the classical communal example. Each herdsman has a single cow, which need 90\$ for feed. The communal can feed two cows, but if the three of them try to feed there, then it is not enough and each herdsman has an additional cost of 50\$.

Initially there is no rule, so that all the players decide to bring their cow to the communal. Hence, the equlibrium state is $(-50, -50, -50)$. 
In order to improve this situation, the following rules are studied:

$R_0:$ No rule.

$R_1:$ A tax (33\$) to any user.

$R_2:$ A raffle (one cannot use it).

$R_3:$ A compensation (66\$ for any one who reject to use it).

\noindent Therefore, the Law is $\Omega=\{R_1,R_2,R_3\}$. Any combination of the aforementioned rules is also possible, hence, we have $2^3=8$ possible regulations, denoted by $G_i\subseteq \Omega$. Under each regulation a new equilibrium state may arise.

We may assume that the  punishment for breaching $R_1$ (i.e., for using the communal without paying the tax) is $g(R_1)=100\$ $, the  punishment for breaching $R_2$  is also $g(R_2)=100\$ $, and for $R_3$ is $g(R_3)=166\$ $. 

The theoretical punishment for breaching a law that is not included in the regulation is 0 (this assumption is needed for the construction of the distance).

Now, we estimate the probability of breaching each law. We may do that for each player or just as an average or mean. Here we did 
 for each player, assuming that one of them is very lawful, the second one no so much, and the third one is a cheat. With the probabilities assumed for each player, the corresponding mean (without weights, since we assume that all the players are equally important (one cow for each herdsman)) is reflected in the following table:

\begin{tabular}{lcccccccc}
Unfulfilled rules   & $A$ & $B$ & $C$ & $D$ & $E$ & $F$ & $G$ & $H$  \\
/ regulation  &   &  &  &  &  &  &  &     \\ \hline

$A=\{\emptyset\}$  & 1 & x & x & x & x & x & x & x   \\  \hline
$B=\{tax\}$  & 22/30 & 8/30 & x & x & x & x & x & x   \\  \hline
$C=\{raffle\}$  & 9/10 & x & 1/10 & x & x & x & x & x   \\  \hline
$D=\{comp.\}$  & 29/30 & x & x & 1/30 & x & x & x & x   \\  \hline
$E=\{tax, raffle\}$  & 82/90 & 5/90 & 0 & x & 3/90 & x & x & x   \\  \hline
$F=\{tax, comp.\}$  & 28/30 & 0 & x & 0 & x & 2/30 & x & x   \\  \hline 
$G=\{raffle, comp.\}$  & 87/90 & x & 0 & 2/90 & x & x & 1/90 & x   \\  \hline
$H=\{tax, raffle, comp.\}$  & 29/30 & 0 & 0 & 0 & 0 & 0 & 0 & 1/30   \\  \hline
\end{tabular}
\medskip

In the table above notice that the probability assigned to $\emptyset$ is the probability of breaching no rule. On the other hand, in some context the probability of breaching a single rule (or a subset of rules) is 0, since it implies to breach another one. For instance, under regulation  $F$, if a player breaks the compensation rule, it means that (s)he is using the communal as well as (s)he is not paying the tax (here we assume that it is buroucratically impossible to ask for compensation as well as pay the tax), thus, (s)he is not breaching the $compeansation$ neither the $tax$, but it is breaking the set of rules $\{tax, compensation\}$. In any case, although the probability of   breaching the single rule is 0, the corresponding punishment exists (at least theoretically, and it will be used when defining distances).

 With the these data, the distances between regulations (as defined in Definition~\ref{Ddis}) have been illustrated in the graph of Figure~\ref{Fgrafo}.
 
\end{example}

 \begin{center}
 
\begin{figure}[H]\label{Fgrafo}
\begin{tikzpicture}
  \graph[n=8,clockwise, radius=7cm,
         empty nodes, nodes={circle,draw}] {
   \foreach \x in {A,...,H} {
      \foreach \y in {\x,...,H} {
        \x -- \y;
      };
    };
    A --["\small{166/30}"] B;
    A --["\small{166/30}"] C;
   A --["{\color{red}{\small{332/30}}}"] D;
    A --["{\color{blue}{\small{80/9}}}", pos=0.6] E; 
    A --["{\color{red}{\small{166/30}}}", pos=0.5] F;
    A --["\small{10}", pos=0.45] G;
   A --["\small{88/10}"] H;
    B --["\small{166/30}"] C;
    B --["\small{166/30}", pos=0.47] D;
    B --["\small{166/30}", pos=0.51] E; 
    B --["{\color{red}{\small{166/30}}}", pos=0.4] F;
    B --["\small{166/30}",pos=0.55] G;
    B --["\small{166/30}", pos=0.45] H;
    C --["\small{1162/30}"] D;
    C --["\small{1298/90}",] E; 
    C --["\small{166/90}", pos=0.55] F;
    C --["{\color{blue}{\small{166/30}}}", pos=0.6] G;
    C --["\small{966/30}",pos=0.45] H;
    D --["\small{1796/90}"] E;
    D --["\small{332/30}", pos=0.45] F; 
    D --["\small{632/30}",pos=0.47] G;
    D --["{\color{red}{\small{332/30}}}", pos=0.55] H;
    E --["\small{1298/90}"] F; 
    E --["\small{170/9}",pos=0.51] G;
    E --["\small{10/3}"] H;
    F --["\small{466/30}"] G;
    F --["\small{966/30}", pos=0.47] H;
   G --["\small{110/3}"] H;
    
  };
  \foreach \x [count=\idx from 0] in {A,...,H} {
    \pgfmathparse{90 + \idx * (360 / 8)}
    \node at (\pgfmathresult:7.4cm) {\x};
  };
  \draw (A) edge[red, thick] node[black,left,pos=.2] {} (D);
  \draw (A) edge[red, thick] node[black,left,pos=.2] {} (F);
  \draw (B) edge[red, thick] node[black,left,pos=.2] {} (F);
  \draw (D) edge[red, thick] node[black,left,pos=.2] {} (H);
  \draw (A) edge[blue, thick] node[black,left,pos=.2] {} (E);
  \draw (C) edge[blue, thick] node[black,left,pos=.2] {} (G);
\end{tikzpicture}
\caption{The graph associated to all the possible regulations for the communal and the corresponding distances.}
\end{figure}
 \end{center}


\section{Applications: law modifications and no-agreement situations}\label{sapply}

Thanks to the mathematical approach to those questions related to regulations and deals, now the corresponding decisions may be chosen by answering optimization problems. In the present section some of these questions and solutions are given.

\subsection{First application: The shortest path, the legal gradiant}\label{sub1}
 In the present section we show how to change the rules of a game in order to achieve a more efficient equilibrium. There may be different manners or paths in order to change the law from a initial regulation to a final one,  but some of them may be better than others when we want to minimize the \emph{social resistence}  to the changes.

As it was shown in the section before, the set of all possible deals is a complete and  weighted and directed graph of $2^{|\Omega|}=|\mathcal{P}(\Omega)|$ nodes, where each node is a deal (i.e. a subset of $\Omega$) and each edge $(G_1,G_2)$ is weighted by the value determined by the function $D(G_2,G_1)$. If $D$ is symmetric, the the graph can be considered undirected.

Therefore, our political problem of legislating while minimizing social resistance is just the well-known problem of finding 
the shortest path  between two vertices   in a graph such that the sum of the weights   is minimized. \cite{graph}

\begin{example}\label{Ea1}
From Figure~\ref{Fgrafo} of Example~\ref{Ecom}, we may observe that, if we want to arrive to regulation $H$ by adding rules one by one from $A$, then any possible path  
 is longer (i.e. worst, since the conflict with respect to the previous situation is bigger) than the (shortest) path $(A, D, G, H)$, whose length is 12'91, whereas the length of the second shortest path $(A, C, E, H)$ is $34'422$ (and the third one would be $(A, B, G, H)$, with length $46'5$). Thus, with the assumed data (mainly, the probabilities and the punishment) the compensation must be the first rule to be introduced.

\end{example}






\subsection{Second application: Searching agreement and consensus}\label{sub3}

When making agreements between $N$ players, each player has a preference on the possible agreements that may be agreed. 
The distance presented may be a tool in order to achieve a deal, arriving to an `intermediate' deal between the preferences of the players.

Returning to the previous case illustrated in Example~\ref{Ea1}, we present the following simple case:

\begin{example}\label{Ea3}
From Figure~\ref{Fgrafo} of Example~\ref{Ecom}, we  may  observe that, for example, if in the region where the communal lies there are two main political parties, such that one of them wants to introduce the regulation $H$ for the communal, whereas the other prefers to continue with no regulation, then, the deal $D$ as well as $G$ could be the midpoint between these two regulations. 
\end{example}

Suppose that in a negotiation there are $N$ players that do not approach positions, and such that each player $n$ has a family of most preferred deals $\{G_{n_k}^{*}\}_{k=1,...,m_n}$ (arguing on the set  of all possible regulations on $\Omega$).  

In this blocking situation, we may wonder which is the minimum distance $r\in \mathbb{R}$
such that  $\bigcap\limits_{n=1}^N (\bigcup\limits_{k=1}^{m_n}B_r(G_{n_k}^* ))\neq \emptyset$,
 as well as it satisfies an \emph{optimal} condition with respect to the player's  preferences, in case they exist. 
Now we are just searching that deal which is as closest as possible to all the most preferred options. Once we know where that meeting point is, then we can apply the study presented in Subsection~\ref{sub1}   in order to make the path from the initial deal to that closest deal.


Now, we assume that each player $A_n$ has a preference $\precsim_n$ (for each $n=1,2,...,N$) on the set of all possible games $\{G\colon G\subseteq \Omega\}=\mathcal{P}(\Omega)$. We assume that each preference $\precsim_n$ is a preorder. The goal of the present section is to achieve an \emph{optimal} deal for disagreement situations. For that purpose, the following concept is defined.

\begin{definition}\rm

Let $\Omega$ be a finite set of possible rules and $A=\{A_n\}_1^N$ a finite set of players with their corresponding preferences $\precsim_n$ on $\mathcal{P}(\Omega)$ (for each $n=1,2,...,N$).  A deal $G^o\subseteq \Omega$ is said to be a \emph{Pareto optimal deal} if for any other deal $G$, $G^o\prec_m G$ implies $G\prec_n G^o$, with $n,m\in \{1,...,N\}$.
\end{definition}

Now, we introduce some interesting concepts for the study of negotiation problems.
\begin{definition}\rm
We shall say that a signer of a deal $G$ is a \emph{$r$-signer} (with respect to a metric $D$), with $r\geq 0$, if she agrees to sign any other deal $G'$ such that the distance from $G$ to $G'$ is smaller than $r$, that is, $D(G',G)<r$.
The value of $r$ is said to be the \emph{deal threshold of the signer}. 

A Game of games $\Gamma$   is said to be a \emph{$r$-step} Game of games if $D(\Psi_1,\Psi_2)\geq r$, for any $\Psi_1, \Psi_2\in \Gamma$.

Given a $r_1$-step Game of games, a $r_2$-signer is said to be \emph{stubborn} whenever $r_2<r_1$.
And it is said to be a \emph{boycotter} if $r_2=0$.   

\end{definition}

Thus, an sttuborn  player is not going to sign any other deal except his/her most preferred options. A boycotter is a player that sign nothing.
 
\begin{example}

In the case illustrated in Example~\ref{Ea1}, 
the length of each step in the  path $(A, D, G, H)$ is at most $166/30$, thus, any player with a deal-threshold bigger than this value would accept the deal. 

On the other hand, there is no path from $A$ to $H$ such that the length of each step is smaller than that value. Thus, consensus would be possible over time whenever all the player are $r$-signers,  with $r\geq 166/30$, for each step carried out at the time. 
\end{example}

The following concept introduces the idea of \emph{compatibility} between the agent's preference and the distance, in the sense that the  closer a l-game is to the maximal elements (with respect to the preference relation), the more preferred it is.

\begin{definition}\rm
Let $\Omega$ be a finite set of possible rules and $A=\{A_n\}_1^N$ a finite set of players with their corresponding preferences $\precsim_n$ on $\mathcal{P}(\Omega)$ (for each $n=1,2,...,N$). For each $n=1,...,N$, we denote by $M_n=\{ {G}_{n_1}^*,...,{G}_{n_{m_n}}^{*} \}$ the maximal elements of $\mathcal{P}(\mathcal{R})$ with respect to $\precsim_n$. Let $D$ be a distance function on $ \mathcal{P}(\Omega)$.
We say that the preferences and the measure $D$ are \emph{compatible} if for any maximal chain $C$ in $(\mathcal{P}(\Omega), \precsim_n)$ and any $G_1, G_2\in {C}$ 
the following condition holds true:
$$ G_1\precsim_n G_2\implies D(G_2, {G}_n^C)\leq D(G_1, {G}_n^C),$$
where ${G}_n^C$ is the maximal element of $C$ with respect to $\precsim_n$.



\end{definition}


\begin{remark}\label{Ralert}
Notice that given a preference $\precsim_n$ and a compatible distance $D$, it may hold that  
$\min\{D(G_{n_k}^*, y)\}_{k=1,...,m_n} < \min\{D(G_{n_k}^*, x)\}_{k=1,...,m_n}$ as well as $y\prec_n x$, for some $x,y\in \mathcal{P}(\Omega)$. 
\end{remark}

In order to achieve that Pareto optimal deal, we will use a compatible distance on the set of games, that is, a distance such that the player's preferences are lower semicontinuous. Anyway, even in case of no preferences, we may apply this method just arguing on distances, since we may assume that the player's preferences are empty and, hence, lower semicontinuous.

First, for each player's preference, we construct the corresponding compatible linear extension.

\begin{definition}\rm
Let $\Omega$ be a finite set of possible rules and $A=\{A_n\}_1^N$ a finite set of players with their corresponding preferences $\precsim_n$ on $\mathcal{P}(\Omega)$ (for each $n=1,2,...,N$).  Assume that these preferences are compatible with respect to a distance $D$.

The \emph{distance linear order} $\preceq_n^D$ for a player  $A_n$ is defined as follows (for any $n=1,...,N$):

$$ x\leq_n^D y \iff  \min\{D(G_{n_k}^*, y)\}_{k=1,...,m_n} \leq \min\{D(G_{n_k}^*, x)\}_{k=1,...,m_n}.   $$


\end{definition}
 The proof of the following proposition is straightforward.
\begin{proposition}
The distance linear order 
is a total preorder.
\end{proposition}


\begin{theorem}

Let $\Omega$ be a finite set of possible rules and $A=\{A_n\}_1^N$ a finite set of players with their corresponding preferences $\precsim_n$ on $\mathcal{P}(\Omega)$ (for each $n=1,2,...,N$). Let $D$ be a distance on $\mathcal{P}(\Omega)$. Let $M_n=\{ {G}_{n_1}^*,...,{G}_{n_{m_n}}^{*} \}$ be the maximal elements of $\mathcal{P}(\Omega)$ with respect to $\precsim_n$.

If $\precsim_n$ is $\tau_U$-lower semicontinuous with respect to the topology defined as 
$$\tau_U=\{ U_r=\displaystyle\bigcup\limits_{k=1}^{k=m_n} B_r({G}_{n_k}^*)\}_{r\geq 0}\cup \{\emptyset\}$$
then
\begin{enumerate}
\item 
 The distance linear order $<_n^D$ is a linear extension of $\prec_n$.
 \item
 The preference $\precsim_n$ and the distance $D$ are compatible.
\end{enumerate}
\end{theorem}
\begin{proof}

\noindent$(1)$
Let $x,y\in X=\mathcal{P}(\Omega)$ such that $x\prec_n y$. Since $\precsim_n$ is $\tau_U$-lower semicontinuous, $y\in U=X\setminus L_{\precsim_n}(x)\in \tau_U$, whereas $x\notin U$. Thus, there exists a radius $r>0$ such that  $y\in U_r=\displaystyle\bigcup\limits_{k=1}^{k=m_n} B_r(G_{n_k}^*)\subseteq U$ as well as $x\notin U_r$. Therefore,  $\min\{D(G_{n_k}^*, y)\}_{k=1,...,m_n} < r \leq \min\{D(G_{n_k}^*, x)\}_{k=1,...,m_n}$ and thus, $x<_n^D y. $
\medskip

\noindent$(2)$ 
Let $x,y\in X=\mathcal{P}(\Omega)$ be such that $x\prec_n y$ and  $G^*\in M_n$   the maximal element of a maximal chain $C$ containing $x$ and $y$. 
By reduction to the absurd, suppose that
there is $r>0$ such that $x\in  B_r(G^*)$ as well as $y \notin  B_r(G^*)$. Then, 
$x\in  \displaystyle\bigcup\limits_{k=1}^{k=m_n} B_r(\overline{G}_{n_k}^*)=U$ whereas $y \notin  \displaystyle\bigcup\limits_{k=1}^{k=m_n} B_r({G}_{n_k}^*)=U$. 
Thus, 
$F=X\setminus U$ is closed and $y\in X\setminus U $. Since the preorder $\precsim_n$ is $\tau_U$-lower semicontinuous, $L_{\precsim_n}(y)$ is contained in $F$ and, hence, $x\in L_{\precsim_n}(y) \subseteq F$. Therefore,  $x\notin \displaystyle\bigcup\limits_{k=1}^{k=m_n} B_r(\overline{G}_{n_k}^*)=U$ and hence, in particular $x\notin B_r(G^*)$ arrived to the desired contradiction.
\end{proof}

\begin{theorem}\label{Tmain}

Let $\Omega$ be a finite set of possible rules and $A=\{A_n\}_1^N$ a finite set of players with their corresponding preferences $\precsim_n$ on $\mathcal{P}(\Omega)$ (for each $n=1,2,...,N$).  Assume that these preferences 
are $\tau_U$-lower semicontinuous.
Let $\leq_n$ be the distance linear extension of $\precsim_n$, for each $n=1,...,N$.

Then, the maximal elements of $\sqsubseteq=\bigcap\limits_{n=1}^N \leq_n$ are all the Pareto optimal deals.
\end{theorem}
\begin{proof}
Let $M$ be a maximal element of  $\bigcap\limits_{n=1}^N \leq_n$. 
Suppose that there is another element $x\in X=\mathcal{P}(\mathcal{R})$ such that $M\prec_n x$ for some $n\in \{1,...,N\}$. Since $<_n$ extends $\prec_n$, then it holds true that $M <_n x$.  On the other hand, $M $ is   a maximal element of  $\bigcap\limits_{n=1}^N \leq_n$, therefore, there is at least one index $k\in\{1,...,N\}$ such that $x\prec_k M$ (otherwise $M\sqsubset x$ and, hence, $M$ is not maximal, arriving to a contradiction). Thus, $M$ is a Pareto optimal deal. 
\end{proof}

Finally, according to  Theorem~\ref{Tmain}, in a disagreement situation where the preferences are $\tau_U$-lower semicontinuous, we may find the set of all Pareto optimal deals. Then, by means of the distance, we may choose that Pareto optimal deal which is also the closest from the most preferred deals, according to each player's preference. This motivates the following definition.

\begin{definition}\rm

Let $\Omega$ be a finite set of possible rules and $A=\{A_n\}_1^N$ a finite set of players with their corresponding preferences $\precsim_n$ on $\mathcal{P}(\Omega)$ (for each $n=1,2,...,N$).  Assume that these preferences are compatible with respect to a distance $D$.
Let $\leq_n$ be the distance linear extension of $\precsim_n$, for each $n=1,...,N$.

A maximal element $G^*$ of $\bigcap\limits_{n=1}^N \leq_n$ is the \emph{Closest Pareto optimal deal} if
$$\max\{D(G_{n_k}^*, G^*)\}_{k=1,...,m_n}\leq \max\{D(G_{n_k}^*, G)\}_{k=1,...,m_n},$$ for any other maximal element $G$  of $\bigcap\limits_{n=1}^N \leq_n$.
\end{definition}


\begin{remark}
In the study above, we have treated all players equally, symmetrically, regardless of their flexibility in reaching agreements. However, it can be approached asymmetrically, building balls of larger or smaller radius depending on the flexibility of each player. This and other questions are left for future works.
\end{remark}
\section{Conclusions}
The main goal of this work is to present a new and original breaking mathematical approach for agreement situations as well as for regulation modifications. The present paper presents the mathematical formalization of several concepts in the field of regulations, negotiations and agreements in the context of a society or a family of agents, as well as a mathematical structure that tries to capture the nature of some of these phenomena. As a consequence,   several problems related to agreements, negotiations, regulations,... may be studied and solved  just as    mathematical optimization problems.

This work is just a first new approach. It gives a  mathematical model for negotiations  and regulation constructions, as well as for comparison of societies. However, this model may  increase its complexity for other types of problems.

\end{document}